\documentclass[a4paper,12pt,thmsa]{amsart}
\usepackage[a4paper,marginratio={1:1},scale={0.72,0.74},footskip=7mm,headsep=10mm]{geometry}

\usepackage{amsfonts}
\usepackage{amssymb,amsmath,latexsym}
\usepackage[dvips]{graphics}
\usepackage{graphicx,subfigure}
\usepackage[dvips]{color}
\usepackage[active]{srcltx}
\usepackage{amsmath}
\usepackage{amsfonts}
\usepackage{amssymb}
\usepackage{psfrag}
\usepackage{color}
\usepackage{url}
\usepackage{amsthm}
\usepackage{array}
\usepackage{pst-tree}
\usepackage{lscape}
\usepackage{dsfont}
\usepackage[utf8]{inputenc}							
\usepackage[T1]{fontenc}   
\usepackage{pstricks,pstricks-add}\usepackage{mathrsfs}  

\renewcommand{\theequation}{\arabic{section}.\arabic{equation}}

\newtheorem{theorem}{Theorem}[section]
\newtheorem{proposition}{Proposition}[section]

\newtheorem{definition}{Definition}[section]
\newtheorem{corollary}{Corollary}[section]

\newtheorem{lemma}{Lemma}[section]

\def\e{\mathbb{E}}
\def\p{\mathbb{P}}

\title[Extinction times of multitype continuous-state branching processes]
{Extinction times of multitype continuous-state branching processes}

\author{Lo\"ic Chaumont}

\author{Marine Marolleau}

\address{Lo\"ic Chaumont -- LAREMA -- UMR CNRS 6093, Universit\'e d'Angers, 2 bd Lavoisier, 49045 Angers cedex~01}

\email{loic.chaumont@univ-angers.fr}

\address{Marine Marolleau -- LAREMA -- UMR CNRS 6093, Universit\'e d'Angers, 2 bd Lavoisier, 49045 Angers cedex~01}

\email{marine.marolleau@univ-angers.fr}



\thanks{}

\date{\today}

\begin{document}

\begin{abstract} 
A multitype continuous-state branching process (MCSBP) ${\rm Z}=({\rm Z}_{t})_{t\geq 0}$, is a 
Markov process with values in $[0,\infty)^{d}$ that satisfies the branching property. 
Its distribution is characterised by its branching mechanism, that is the data of $d$ Laplace exponents 
of $\mathbb{R}^d$-valued spectrally positive L\'evy processes, each one having $d-1$ increasing components. 
We give an expression of the probability for a MCSBP to tend to 0 at infinity in term of its branching 
mechanism. Then we prove that this extinction holds at a finite time if and only if some condition bearing 
on the branching mechanism holds. This condition extends Grey's condition that is well known for $d=1$. 
Our arguments bear on elements of fluctuation theory for spectrally positive additive L\'evy fields 
recently obtained in \cite{cma1} and an extension of the Lamperti representation in higher dimension 
proved in \cite{cpgub}.  
\end{abstract}

\maketitle

\section{Introduction}

A multitype continuous state branching process (MCSBP), ${\rm Z}$, with family of probability 
measures  $\mathbb{P}_{\rm r}$, ${\rm r}\in\mathbb{R}^{d}_{+}$, $d\geq 1$, is a 
$[0,\infty)^{d}$-valued Markov process, satisfying the \emph{branching property} : 
\[\mathbb{E}_{{\rm r}_{1}+{\rm r}_{2}}(e^{-\langle\lambda,{\rm Z}_{t}\rangle}) = 
\mathbb{E}_{{\rm r}_{1}}(e^{-\langle\lambda,{\rm Z}_{t}\rangle})
\mathbb{E}_{{\rm r}_{2}}(e^{-\langle\lambda,{\rm Z}_{t}\rangle}),\;\;
\lambda,{\rm r}_{1},{\rm r}_{2}\in\mathbb{R}^{d}_{+}.\]
MCSBP's were first introduced in the late 60's by Watanabe \cite{wa} and there has been renewed interest in 
such processes in more recent years with the works of Duffie, Filipovi\'c and Schachermayer \cite{dfs},
Barczy, Li and Pap \cite{blp}, Caballero, P\'erez Garmendia and Uribe Bravo \cite{cpgub}, Kyprianou and Palau
\cite{kp},...

Asymptotic behavior and extinction times are among the most studied questions in recent times.
In the articles \cite{kp} and \cite{kpr}, Kyprianou, Palau and Ren provide a counterpart to the well known case of 
multitype Galton Watson processes by proving that the asymptotic behaviour of the MCSBP ${\rm Z}$ is characterized 
by the value of the Perron-Frobenius eigenvalue of the mean matrix denoted here by $\rho$. 
More specifically, extinction occurs if and only if $\rho\leq 0$. In this case, 
${\rm Z}$ is said to be critical ($\rho=0$) or sub-critical ($\rho<0$). When $\rho>0$, the process is said to be
super-critical and 
${\rm Z}$ either becomes extinct with positive probability or has exponential growth under a log-condition. 
We will see here that, as in dimension one, extinction in continuous time differs from the discrete case. 
Indeed, in addition to the possibility of becoming extinct in a finite time, that is there exists $t\geq 0$ such 
that for all $i\in[d]$, $Z^{(i)}_{t}=0$, the process can be extinguished at infinity, that is for all $t\geq 0$, 
there exists $i\in[d]$ such that $Z^{(i)}_{t}>0$ and $\lim\limits_{t\rightarrow +\infty} {\rm Z}_{t}=0$.
When $d=1$, Grey's condition, see \cite{gr}, gives a necessary and sufficient condition for the continuous state 
branching process $Z$ to become extinct at a finite time in terms of its branching mechanism, $\varphi$. More 
precisely, $Z$ becomes extinct at a finite time if and only if 
\begin{equation}\label{9930}
\int^\infty\frac{ds}{\varphi(s)}<\infty\,.
\end{equation}
However, when $d>1$, we do not know how to distinguish these two 'types' of extinction.
The aim of this article is to provide necessary and sufficient conditions bearing on the mechanism of ${\rm Z}$ 
for extinction to take place at a finite time, in order to extend Grey's condition (\ref{9930}) to the multitype 
case.

One of the major results used for our proof is the extension to the multitype case of the Lamperti representation,
recently obtained in \cite{cpgub}. This result asserts that ${\rm Z}$ can be represented as the unique solution 
of the following equation,
\[(Z^{(1)}_{t},...,Z^{(d)}_{t})={\rm r} + \left(\sum\limits_{j=1}^{d} X^{1,j}_{\int\limits_{0}^{t}Z^{(j)}_{s}ds},...,
\sum\limits_{j=1}^{d}X^{d,j}_{\int\limits_{0}^{t}Z^{(j)}_{s}ds}\right), \;\; t\geq 0 ,\] 
where ${\rm X}^{(j)}= {}^{t}(X^{1,j},\dots, X^{d,j})$, $j\in[d]$ are $d$ independent Lévy processes such that for 
all $j\in[d]$, $X^{j,j}$ is a spectrally positive Lévy process, that is with no negative jumps, and for all 
$i\neq j$, $X^{i,j}$ is a subordinator (here ${}^t{\rm u}$ means the transpose of 
the vector ${\rm u}\in\mathbb{R}^d$). The right hand side of the above equation can be re-written as 
${\rm r}+\mathbf{X}\left(\int\limits_{0}^{t} {\rm Z}_{s}ds\right)$, where $\{\mathbf{X}_{\rm t}, 
{\rm t}\in\mathbb{R}^{d}_{+}\}$ is the spectrally positive additive Lévy field (spaLf) defined as the sum of the 
$d$ independent Lévy processes ${\rm X}^{(j)}$, that is 
\[\mathbf{X}_{\rm t}:= {\rm X}^{(1)}_{t_{1}}+\dots +{\rm X}^{(d)}_{t_{d}}, \;\; {\rm t}\in\mathbb{R}^{d}_{+}\,.\ \]
Note that spaLf's have been defined and studied in a previous paper under the aspect of fluctuation theory, 
see \cite{cma1}. When $d=1$, this result is due to Lamperti who proved, in the $1970$'s, that any time continuous 
branching process can be represented as a spectrally positive Lévy process time changed by the inverse of some 
integral functional. Note that when $d=1$, the Lamperti representation is constructive, that is the branching 
process can be made explicit from the leading L\'evy process, whereas this is not the case in higher dimension. 

Our main result asserts that if the Laplace exponent $\tilde{\varphi}_i$ of each diagonal L\'evy process 
$X^{i,i}$ in the above Lamperti representation satisfies Grey's condition, that is if
\begin{equation}\label{9340}
\int^\infty\frac{ds}{\tilde{\varphi}_i(s)}<\infty,\;\;\;\mbox{for all $i\in[d]$}\,,
\end{equation}
then extinction can only occur at a finite time. If (\ref{9340}) is not satisfied 
for some $i\in[d]$, then ${\rm Z}$ can only be extinguished at infinity. Moreover, when starting from 
${\rm r}\in\mathbb{R}_+^d$, the probability that ${\rm Z}_t$ tends to $\bf 0$, 
when $t$ tends to $\infty$ is $e^{-\langle {\rm r},\phi(\mathbf{0})\rangle}$,
where $\phi$ is the inverse of the Laplace exponent of the spaLf ${\bf X}$. In particular, ${\rm Z}$ becomes 
extinct almost surely if and only if it is critical or sub-critical.

This result indicates in particular that the off-diagonal subordinators $X^{i,j}$, $i\neq j$ have no influence 
on the nature of the extinction. It may appear rather counter-intuitive. However, thinking of the neutral case  
helps us to understand this phenomenon. Indeed, in the neutral case, that is when the law of $\sum_{i=1}^dX^{i,j}$
does not depend on $j$, the MCSBP ${\rm Z}$ behaves like a single type continuous state branching process. 
More specifically, $Z^{(1)}+\dots+Z^{(d)}$ is a continuous state branching process  whose branching mechanism is
$d\hat{\varphi}$, where $\hat{\varphi}$ is the Laplace exponent  of $\sum_{i=1}^dX^{i,j}$. But provided 
the $X^{i,i}$'s are not subordinators, which is excluded here, $\int^\infty\frac{ds}{\hat{\varphi}(s)}<\infty$ 
if and only if (\ref{9340}) holds for each $i$, see the discussion after Theorem \ref{4588}. Therefore, from 
Grey's condition (\ref{9930}), ${\rm Z}$ becomes extinct at a finite time with positive 
probability if and only if (\ref{9340}) is statisfied. 

The next section is devoted to some reminders and preliminary results about MCSBP's and spaLf's. Then in Section 
\ref{1300} we state our main results and in Sections \ref{proofs.ext}, \ref{proofs.Grey.M} and the Appendix we 
give proofs of these results. 

\section{Preliminary results on MCSBP's and spaLf's}\label{Reminders}

We use the notation $\mathbb{R}_+=[0,\infty)$ and $[d]=\{1,\dots,d\}$, where $d\ge1$ is an integer. 
Vectors of $\mathbb{R}^d$ will be written in roman characters and their coordinates in italic, as 
follows: ${\rm x}=(x_1,\dots,x_d)$ and ${}^t{\rm x}={}^t(x_1,\dots,x_d)$ will denote the transpose of ${\rm x}$. 
The $i$-th unit vector of $\mathbb{R}_+^d$ will be denoted ${\rm e}_i=(0,\dots,0,1,0,\dots,0)$. We set 
${\bf 0}=(0,0,\dots,0)$ and ${\bf 1}=(1,1,\dots,1)$ respectively for the zero vector of $\mathbb{R}_+^d$ 
and the  vector of $\mathbb{R}_+^d$ whose all coordinates are equal to 1. For ${\rm s}=(s_1,\dots,s_d)$ and 
${\rm t}=(t_1,\dots,t_d)\in\mathbb{R}_+^d$, we write ${\rm s}\le{\rm t}$ if $s_i\le t_i$ for all $i\in[d]$ and 
we write ${\rm s}<{\rm t}$ if ${\rm s}\le {\rm t}$ and there exists $i\in[d]$ such that $s_i<t_i$.  We will 
denote by $\langle {\rm x},{\rm y}\rangle$, ${\rm x},{\rm y}\in\mathbb{R}^d$ the usual inner product on 
$\mathbb{R}^d$ and by $|{\rm x}|$ the Euclidian norm of ${\rm x}$.\\ 

\subsection{Basics on MCSBP's}\label{9522} Let $\delta$ be an element external to $\mathbb{R}_+^d$ and 
$R=\mathbb{R}_+^d\cup\{\delta\}$ be the one point Alexandrov compactification of $\mathbb{R}_+^d$. An $R$-valued 
strong Markov process 
${\rm Z}=\{({\rm Z}_t)_{t\ge0},(\mathbb{P}_{{\rm r}})_{{\rm r}\in\mathbb{R}^{d}_{+}}\}$ with $\delta$ as a trap 
is called a multitype (or a $d$-type) continuous state branching process (MCSBP) if it satisfies,
\begin{equation}\label{8422}
\mathbb{E}_{{\rm r}_{1}+{\rm r}_{2}}[e^{-\langle \lambda,{\rm Z}_{t}\rangle}] = 
\mathbb{E}_{{\rm r}_{1}}[e^{-\langle \lambda,{\rm Z}_{t}\rangle}]
\mathbb{E}_{{\rm r}_{2}}[e^{-\langle \lambda,{\rm Z}_{t}\rangle}], \; t\geq 0, \;
{\rm r}_{1},{\rm r}_{2},\lambda\in\mathbb{R}^{d}_{+}\,,
\end{equation}
where by convention $e^{-\langle \lambda,\delta\rangle}=0$, for all $\lambda\in\mathbb{R}^{d}_{+}$. 
The property (\ref{8422}) is called the {\it branching property} of ${\rm Z}$. We recall that $\mathbf{0}$ is 
an absorbing state for ${\rm Z}$ and that it is the only absorbing state other than $\delta$. 
The branching property of ${\rm Z}$ implies directly the existence of a mapping 
$(t,\lambda)\mapsto{\rm u}_t(\lambda)=(u^{(1)}_t(\lambda),\dots,u^{(d)}_t(\lambda))$ satisfying 
\begin{equation}\label{expLaplaceZ}
\mathbb{E}_{\rm r}[e^{-\langle \lambda,{\rm Z}_{t}\rangle}]
	= e^{-\langle {\rm r},{\rm u}_{t}(\lambda)\rangle},\;\; t\geq 0,\;\lambda,{\rm r}\in\mathbb{R}^{d}_{+}.
\end{equation}
From the  Markov property of ${\rm Z}$, we derive the {\it semigroup property} of ${\rm u}_t(\lambda)$,
\[{\rm u}_{t+s}(\lambda)={\rm u}_t({\rm u}_s(\lambda)),\;\;\;\mbox{for all $s,t\ge0$ and 
$\lambda\in\mathbb{R}^d_+$.}\]
Moreover for each fixed $t$, ${\rm u}_t(\lambda)$ is differentiable 
in $\lambda$. It was first proved in \cite{wa}, see also Proposition 6.1 in \cite{dfs}, that for each 
$\lambda\in\mathbb{R}_+^d$, $t\mapsto{\rm u}_t(\lambda)$ is differentiable in $t$ and is the unique solution 
to the following differential system,
\[(S_{\varphi}) \qquad\qquad\;
 \left\lbrace
	\begin{array}{ll}
		\dfrac{\partial}{\partial t} u^{(i)}_{t}(\lambda) = -\varphi_{i} ({\rm u}_{t}(\lambda))
		\\[6pt] u^{(i)}_{0}(\lambda) =\lambda_{i}
	\end{array}
	\right., 
	\;\;\; i\in[d],\, t\geq 0,\]
where each $\varphi_i$ is the Laplace exponent of a possibly killed $d$-dimensional L\'evy process 
${\rm X}^{(i)}={}^t(X^{1,i},\dots,X^{d,i})$. The coordinate $X^{i,i}$ is a (possibly killed) spectrally positive 
Lévy process and $X^{i,j}$, $j\neq i$ are (possibly killed) subordinators. More specifically, we define a 
probability measure $\p$ under which the processes ${\rm X}^{(i)}$, $i\in[d]$ are $d$ independent possibly killed 
$d$-dimensional L\'evy processes. Then for each $i\in[d]$, $\varphi_i$ is defined by 
\[\e[e^{-\langle \lambda, {\rm X}^{(i)}_t\rangle}]=e^{t\varphi_i(\lambda)}\,,\;\;\;t\ge0\,,\;\;\;{\bf\lambda}=
(\lambda_1,\dots,\lambda_d)\in \mathbb{R}_+^d\,.\]
Its L\'evy-Khintchine decomposition is written as,
\begin{equation}\label{2552}
\varphi_i(\lambda)=-\alpha_i-\sum_{j=1}^d a_{j,i}\lambda_j+\frac12q_i\lambda_i^2-
\int_{\mathbb{R}_+^d}(1-e^{-\langle\lambda, {\rm x}\rangle}-\langle\lambda, {\rm x}\rangle1_{\{|{\rm x}|<1\}})\,
\pi_i(d{\rm x})\,, \;\;\lambda\in\mathbb{R}^{d}_{+},
\end{equation}
where $\alpha_i\ge0$, $(a_{j,i})_{i,j\in[d]}$ is an essentially nonnegative matrix i.e.~$a_{j,i}\ge0$ for $i\neq j$, 
$q_i\ge0$ and $\pi_i$ is a measure on $\mathbb{R}_+^d$ such that $\pi_i(\{\mathbf{0}\})=0$ and
\[\int_{\mathbb{R}_+^d}\left[(1\wedge |{\rm x}|^2)+\sum\limits_{j\neq i}(1\wedge x_j)\right]\pi_j(d{\rm x})<\infty\,.\]
The mapping $\varphi=(\varphi_{1},\dots,\varphi_{d})$ is called the \emph{branching mechanism} of the MCSBP ${\rm Z}$.
We emphasize that for each $i\in[d]$, $\varphi_i$ is a convex function. Moreover, for all $j\neq i$ and 
$\lambda_1,\dots,\lambda_{j-1}$, $\lambda_{j+1},\dots,\lambda_d$ the function $\lambda_j\mapsto \varphi_i(\lambda)$ 
is non increasing and whenever $X^{i,i}$ is not a subordinator, the function $\lambda_i\mapsto \varphi_i(\lambda)$
tends to $\infty$ as $\lambda_i\rightarrow\infty$. Conversely, still according to \cite{wa}, for each mapping such 
as $\varphi$ in (\ref{2552}) there exists a unique (in law) MCSBP 
${\rm Z}=\{({\rm Z}_t)_{t\ge0},(\mathbb{P}_{{\rm r}})_{{\rm r}\in\mathbb{R}^{d}_{+}}\}$ 
with branching mechanism $\varphi$. Note that MCSBP's belong to a more general class of processes called affine 
processes. Existence and uniqueness of the solution of equation $(S_\varphi)$ is proved in this general setting 
in \cite{dfs}.\\

A MCSBP ${\rm Z}$ is said to be \emph{conservative} if for all ${\rm r}\in\mathbb{R}_+^d$ and $t>0$, 
$\p_{\rm r}({\rm Z}_t\in\mathbb{R}_+^d)=1$. From (\ref{expLaplaceZ}) this can be expressed in terms of the 
Laplace exponent of ${\rm Z}$ as $e^{\langle {\rm r},{\rm u}_t(\mathbf{0})\rangle}=1$, for all 
${\rm r}\in\mathbb{R}_+^d$ and $t>0$, so that ${\rm Z}$ is conservative if and only if 
\begin{equation}\label{2899}
u_t^{(i)}(\mathbf{0})=0,\;\;\mbox{for all $t>0$ and $i\in[d]$.}
\end{equation}
When $d=1$, it is proved in \cite{gr} that ${\rm Z}$ is conservative if and only if its mechanism $\varphi$ 
satisfies
\[\int_{0+}\dfrac{d\alpha}{|\varphi(\alpha)|}=\infty\,.\]
When $d>1$, finding necessary and sufficient conditions on the branching mechanism for ${\rm Z}$ to 
be conservative is an open question which is not our purpose here. Note however that sufficient conditions 
have been given in Lemma 9.2 of \cite{dfs}. It is also possible to construct examples from the neutral case
defined in the introduction, that is when the law of $\sum_{i=1}^dX^{i,j}$ does not depend on $j$, since 
in this case, the branching mechanism of ${\rm Z}$ is $d\hat{\varphi}$, where $\hat{\varphi}$ is the 
Laplace exponent of $\sum_{i=1}^dX^{i,j}$. Let us also note that if for some $i\in[d]$, 
$u^{(i)}_t(\mathbf{0})=0$, for all $t\ge0$, then from $(S_{\varphi})$, for all $t\ge0$, 
\begin{eqnarray*}
\frac{\partial}{\partial t}u_t^{(i)}(\mathbf{0})&=&-\varphi_i(u_t^{(1)}(\mathbf{0}),\dots,
u_t^{(i-1)}(\mathbf{0}),0,u_t^{(i+1)}(\mathbf{0}),
\dots,u_t^{(d)}(\mathbf{0}))\\
&=&0.
\end{eqnarray*}
But $\varphi_i$ is non increasing on all coordinates $j\neq i$, so that
\[-\varphi_i(u_t^{(1)}(\mathbf{0}),\dots,u_t^{(i-1)}(\mathbf{0}),0,u_t^{(i+1)}(\mathbf{0}),\dots,
u_t^{(d)}(\mathbf{0}))\ge\alpha_i\] 
and hence $\alpha_i=0$, see also Proposition 9.1 in \cite{dfs}. Therefore, if ${\rm Z}$ is conservative, then 
for all $i\in[d]$, $\alpha_i=0$. We will assume throughout this paper that ${\rm Z}$ is conservative.\\

There also exists a pathwise connection between MCSBP's and their branching mechanism which will be of great 
use for our results. For $d=1$, this has been known for a long time as the Lamperti representation, see 
\cite{la} and \cite{club}. This representation has 
recently been extended to the multitype case in \cite{gt} and \cite{cpgub}, see also \cite{ch} for discrete 
valued processes. More specifically, let ${\rm r}\in\mathbb{R}_+^d$ and ${\rm X}^{(i)}$, $i\in[d]$ be $d$ 
independent L\'evy processes whose Laplace exponents have the form (\ref{2552}). Then Theorem 1 in 
\cite{cpgub} asserts that there exists a unique solution to the equation,
\begin{equation}\label{Lamperti}
Z^{(i)}_{t}
	=r_{i} + \sum\limits_{j=1}^{d} X^{i,j}\left(\int\limits_{0}^{t}Z^{(j)}_{s}ds\right), \;\; t\geq 0,\;\;i\in[d].
\end{equation}
Moreover, if $\p_{\rm r}$ denotes the law of the solution ${\rm Z}_t=(Z^{(1)}_{t},\dots,Z^{(d)}_{t})$, $t\ge0$, 
then ${\rm Z}=\{({\rm Z}_t)_{t\ge0},(\mathbb{P}_{{\rm r}})_{{\rm r}\in\mathbb{R}^{d}_{+}}\}$ is a MCSBP with 
branching mechanism $\varphi$. Conversely, every MCSBP can be obtained from this construction. Note that the 
paper \cite{cpgub} actually treats the more general setting of affine processes.

\subsection{Some reminders on spaLf's}\label{8225} 
Whether characterized analytically through the differential system 
$(S_\varphi)$ or path by path through 
equations $(\ref{Lamperti})$, MCSBP's require a good knowledge of the underlying L\'evy processes 
${\rm X}^{(j)}$, $j\in[d]$ in order to be properly studied. 
Note that in equation (\ref{Lamperti}), these 
L\'evy processes live in different time scales and that the leading process is actually the multivariate 
stochastic field
\[\mathbf{X}_{\rm t}:={\rm X}_{t_1}^{(1)}+\dots+{\rm X}^{(d)}_{t_d}=\left(\sum_{j=1}^dX^{i,j}_{t_j}\right)_{i\in[d]}
\,,\;\;\;\mbox{for}\;\;\;{\rm t}=(t_1,\dots,t_d)\in\mathbb{R}_+^d\,.\] 
This process is called a \emph{spectrally positive additive L\'evy field} (spaLf) and has been studied in 
\cite{cma1} from the aspect of fluctuation theory. The mapping $\varphi$ defined above is actually the Laplace 
exponent of this spaLf, that is
\[\e[e^{-\langle\lambda, \mathbf{X}_{\rm t}\rangle}]=e^{\langle {\rm t},\varphi(\lambda)\rangle}\,,\;\;\;{\rm t},
\lambda\in \mathbb{R}_+^d\,.\]
Similarly to the case $d=1$, the construction of the solution of (\ref{Lamperti}) only requires the paths of 
$\mathbf{X}$ up to its first hitting time at level $-{\rm r}$. This notion of first hitting times for spaLf's 
has been defined in \cite{cma1} as the smallest solution ${\rm s}=(s_1,\dots,s_d)$ to the system 
\begin{equation}\label{1886}
\sum_{j=1}^dX^{i,j}_{s_j}=-r_i,\;\;\;i\in[d]_{\rm s}\,,
\end{equation}
where 'smallest' is to be understood in the usual partial order of $\mathbb{R}^d$ 
and $[d]_{\rm s}=\{i\in[d]:s_i<\infty\}$, see the appendix. We will denote by 
$\mathbf{T}_{\rm r}=(T_{\rm r}^{(1)},\dots,T_{\rm r}^{(d)})$ this solution and we will refer to it as 
\emph{the $($multivariate$)$ first hitting time of level $-{\rm r}$ by the spaLf
$\mathbf{X}=\{\mathbf{X}_{\rm t}, {\rm t}\in\mathbb{R}^{d}_{+}\}$}. In a purely formal way we may set,
\begin{equation}\label{2678}
\mathbf{T}_{\rm r}=\inf\{{\rm t}:\mathbf{X}_{{\rm t}}=-{\rm r}\}\,.
\end{equation}
Note that from \cite{cma1}, $\p(\mathbf{T}_{\rm r}\in\mathbb{R}_+^d)+\p(T_{\rm r}^{(i)}=\infty,i\in[d])=1$, 
that is, with probability one, either all coordinates of $\mathbf{T}_{\rm r}$ are finite or all of 
them are infinite.\\ 

Let us now state two important hypothesis which will be in force throughout this paper. The first one is,
\[(H_{\varphi}) \qquad\qquad\; D_{\varphi} =\{ \lambda\in\mathbb{R}^{d}_{+} : 
\varphi_{j}(\lambda)>0, j\in[d]\} \;\text{is not empty.}\]
Assuming $(H_{\varphi})$ simply allows us to exclude the spaLf's which do not hit any level, almost surely, 
see Theorem \ref{propTr} below. 
When $d=1$ this boils down to exclude subordinators. Our second assumption regards the mean matrix  
$M =\left(m_{ij}\right)_{i,j\in[d]}$, where $m_{ij}:=\e(X_1^{i,j})=-\lim\limits_{\lambda\rightarrow0}
\dfrac{\partial}{\partial \lambda_{i}}\varphi_{j}(\lambda)$. We say that $M$ is {\it irreducible} if 
for all $i,j\in[d]$, there exist $n\geq 1$ and some indices $i=i_{1},\dots, i_{n}=j$ such that for all 
$k\in [n-1]$, $m_{i_{k}i_{k+1}}\neq 0$. Note that for all $i,j\in[d]$, $m_{ij}\in(-\infty,\infty]$.
A MCSBP ${\rm Z}$ is said to be irreducible if the mean matrix $M$ of the underlying spaLf ${\bf X}$ in 
the Lamperti representation is irreducible. We  will also sometimes say that the spaLf ${\bf X}$ is 
irreducible. When ${\bf X}$ is integrable, that is when $m_{ij}<\infty$, for all $i,j\in[d]$, 
Perron-Frobenius theory asserts that $M$ has a real eigenvalue $\rho$ with multiplicity equal to 1 and 
such that the real part of any other eigenvalue is less than $\rho$. In this case the MCSBP ${\rm Z}$ 
will be said that it is subcritical, critical or super critical according as $\rho<0$, $\rho=0$ or 
$\rho>0$. 

The following result is proved in \cite{cma1}, see Proposition 3.1, Theorems 3.3 and 3.4 therein.

\begin{theorem}\label{propTr}
Let $\mathbf{X}$ be a spaLf such that $\alpha_i=0$, for all $i\in[d]$.  For ${\rm r}\in\mathbb{R}_+^d$, 
let $\mathbf{T}_{\rm r}$ be its first hitting time of level $-{\rm r}$ as defined above. Then,
\begin{itemize}
\item[$1.$] $\mathbb{P}(\mathbf{T}_{\rm r}\in\mathbb{R}^{d}_{+})>0$ for some $($and hence for all$)$ 
${\rm r}\in\mathbb{R}^{d}_{+}$ if and only if $(H_{\varphi})$ holds.
\item[$2.$] If $(H_{\varphi})$ holds then there is a mapping 
$\phi=(\phi_1,\dots,\phi_d):\mathbb{R}_+^d\rightarrow[0,\infty)^d$ such that 
\begin{equation}\label{5147}
\e[e^{-\langle\lambda,\mathbf{T}_{\rm r}\rangle}]=
e^{-\langle {\rm r},\phi(\lambda)\rangle},\;\;\; \lambda\in\mathbb{R}_+^d\,.
\end{equation}
Moreover, for all $\lambda\in(0,\infty)^d$, $\phi(\lambda)\in D_{\varphi}$ and the mapping 
$\phi:(0,\infty)^{d}\rightarrow D_{\varphi}$ is a diffeormorphism whose inverse is
$\varphi:D_{\varphi}\rightarrow (0,\infty)^{d}$, that is 
\[\varphi(\phi(\lambda))=\lambda,\;\;\;\lambda\in(0,\infty)^d.\]
\item[$3.$] If ${\bf X}$ is irreducible and if $(H_{\varphi})$ holds, then the values $\mathbf{0}$ and 
$\phi(\mathbf{0})$ are the only solutions of the equation $\varphi(\lambda)=\mathbf{0}$, 
$\lambda\in\mathbb{R}_+^d$. 
Moreover, either $\phi(\mathbf{0})=\mathbf{0}$ or $\phi(\mathbf{0})\in(0,\infty)^{d}$.
\end{itemize}
\end{theorem}
\noindent 
An immediate consequence of Theorem \ref{propTr} is the following expression of the probability for the 
first hitting time to be finite on each coordinate, 
\begin{equation}\label{2679}
\mathbb{P}(\mathbf{T}_{\rm r}\in\mathbb{R}^{d}_{+})= e^{-\langle {\rm r},\phi(\mathbf{0})\rangle}\,.
\end{equation}
When $d=1$ we derive directly from the Lamperti representation (\ref{Lamperti}) the almost sure 
equality $\mathbf{T}_{\rm r}=\int_0^\infty {\rm Z}_s\,ds$, on the set $\{{\rm Z}_t\in\mathbb{R}_{+},\,t\ge0\}$. 
The general case is much more delicate 
to deal with. Indeed, the underlying spaLf being multi-indexed, it may reach the level $-{\rm r}$ through
infinitely many different paths. However, the representation of $\mathbf{T}_{\rm r}$ for $d\ge1$ is as 
expected and the following result will be proved in the Appendix. 

\begin{proposition}\label{8260}
Let ${\rm Z}$ be a MCSBP and let ${\bf X}$ be the underlying spaLf in the Lamperti representation 
$(\ref{Lamperti})$. Then for all $i\in[d]$, and ${\rm r}\in\mathbb{R}_+^d$,
\[T_{\rm r}^{(i)}=\int_0^\infty Z_s^{(i)}\,ds\,,\;\;\mbox{$\p_{\rm r}$-a.s.~on the set 
$\{{\rm Z}_t\in\mathbb{R}^d_{+},\,t\ge0\}$}\,.\]
\end{proposition}  

Let us finally specify that since the results of this paper are only concerned with distributional properties 
of MCSBP's, when refereing to such a process ${\rm Z}$, there will be no ambiguity in mentioning the 
underlying spaLf ${\bf X}$ in the Lamperti representation (\ref{Lamperti}). Then $\p$ will be a reference 
probability measure under which ${\bf X}$ is a spaLf issued from 0 and ${\rm Z}$ is a MCSBP with family 
of probability measures $(\mathbb{P}_{{\rm r}})_{{\rm r}\in\mathbb{R}^{d}_{+}}$. 

\section{Main results}\label{1300}

We start by fixing some natural conditions for the study of extinction times. We have already mentioned that 
we assume conservativeness. When ${\rm Z}$ is not irreducible, several subsets of types (that is sub-vectors 
of ${\rm Z}$) can have an asymptotic behaviour that does not depend on other types. In such 
cases, the study of extinction of ${\rm Z}$ is reduced to the study of partial extinction of its irreducible 
classes. Therefore in order to study the extinction property, it is natural to assume that ${\rm Z}$ is irreducible. 
On the other hand, recall that from Theorem \ref{propTr} if $(H_\varphi)$ does not hold, then the 
underlying spaLf satisfies $\p({\bf T}_{\rm r}\in\mathbb{R}_+^d)=0$, for all ${\rm r}\in\mathbb{R}_+^d$
and we can derive from the Lamperti representation (\ref{Lamperti}), that
$\p_{\rm r}\left(\lim\limits_{t\rightarrow\infty}{\rm Z}_t=\mathbf{0}\right)=0$, that is, with probability one, 
extinction does not hold. Therefore, throughout the rest of this paper, 
we will assume that 
\[\mbox{\it ${\rm Z}$ is conservative, irreducible and $(H_\varphi)$ holds.}\]
In particular, these assumptions will not be recalled in the statements. Then, recall from Theorem \ref{propTr} 
that the branching mechanism $\varphi:D_{\varphi}\rightarrow (0,\infty)^{d}$ of ${\rm Z}$ admits an inverse denoted 
by $\phi$. 

We first state the expression of the probability of extinction in terms of the (inverse of) the branching 
mechanism.

\begin{theorem}\label{1005}
Let ${\rm Z}$ be a MCSBP, then for all ${\rm r}\in\mathbb{R}_+^d$, 
\begin{equation}\label{5127}
\p_{\rm r}\left(\lim\limits_{t\rightarrow\infty}{\rm Z}_t=\mathbf{0}\right)=
e^{-\langle {\rm r},\phi(\mathbf{0})\rangle}.
\end{equation}
\end{theorem}
\noindent Recall that when the underlying spaLf is integrable, that is $\e(X_1^{i,j})<\infty$, for all 
$i,j\in[d]$, then from Theorem 3.4 in \cite{cma1}, $\phi(0)>0$ if and only if ${\rm Z}$ is supercritical.
Hence our result is consistent with Theorem 2 and the subsequent remark in \cite{kp} where it is claimed
that if ${\rm Z}$ is critical or subcritical then extinction (at a finite time or at infinity) holds 
almost surely.

We define the probability of extinction at a finite time under $\p_{\rm r}$, ${\rm r}\in\mathbb{R}_+^d$ by,
\begin{equation}\label{5347}
q_{\rm r}:=\p_{\rm r}(\mbox{${\rm Z}_{t}=\mathbf{0}$ for some $t>0$}).
\end{equation}
Since the family of events $\{{\rm Z}_t=\mathbf{0}\}$ is non decreasing in $t$ and converges to the event 
$\{\mbox{${\rm Z}_{t}=\mathbf{0}$ for some $t>0$\}}$, 
the probability $q_{\rm r}$ can be written as 
\begin{equation}\label{5627}
q_{\rm r}=\lim\limits_{t\rightarrow +\infty} \mathbb{P}_{\rm r}({\rm Z}_{t}=\mathbf{0})\,.
\end{equation}
Note also that from (\ref{expLaplaceZ}), for all $t\ge0$, the mapping $\lambda\mapsto{\rm u}_t(\lambda)$ 
statisfies ${\rm u}_t(\lambda)\le{\rm u}_t(\lambda')$, whenever $\lambda\le \lambda'$. Moreover, the limit 
${\rm u}_t(\infty)$ of ${\rm u}_t(\lambda)$ when all coordinates $\lambda_1,\dots,\lambda_d$ of $\lambda$ 
tend to $\infty$ does not depend on the relative speeds at which the $\lambda_i$'s tend to $\infty$ and 
it satisfies,
\begin{equation}\label{2782}
\p_{\rm r}({\rm Z}_t=\mathbf{0})=e^{-\langle {\rm r},{\rm u}_t(\infty)\rangle}\,.
\end{equation}
The probability of extinction $q_{\rm r}$ depends on the finiteness of ${\rm u}_t(\infty)$ as follows. 

\begin{theorem}\label{5003} Let ${\rm Z}$ be a MCSBP. Assume that none of the processes $X^{i,j}$ for 
$i,j\in[d]$ such that $i\neq j$ is a compound Poisson process. Then we have the following dichotomy$:$ 
\begin{itemize}
\item[$(i)$] either for all $t\geq 0$ and $i\in[d]$, $u^{(i)}_{t}(\infty)<\infty$ and then for all
$t>0$ and ${\rm r}\in\mathbb{R}_+^d$, $\p_{\rm r}({\rm Z}_t=\mathbf{0})>0$ and in this case, 
\[q_{\rm r}=e^{-\langle {\rm r},\phi(\mathbf{0})\rangle},\;\;\;\mbox{for all ${\rm r}\in\mathbb{R}_+^d$, }\]
\item[$(ii)$] or for all $t\geq 0$ and $i\in[d]$, $u^{(i)}_{t}(\infty)=\infty$ and then for all 
${\rm r}\in\mathbb{R}^{d}_{+}\setminus\{\mathbf{0}\}$, $q_{\rm r}=0$.
\end{itemize}
\end{theorem}

\noindent Our goal is to distinguish between two types of extinction defined by the two exhaustive events
\[\left\{\lim\limits_{t\rightarrow\infty}{\rm Z}_t=\mathbf{0} \,\text{ and }\, {\rm Z}_t>\mathbf{0},\;
\mbox{for all $t>0$}\right\}\;\;\;\mbox{and}
\;\;\;\left\{\mbox{${\rm Z}_{t}=\mathbf{0}$ for some $t>0$}\right\}.\]
In the first case, we will say that extinction occurs at infinity and that the MCSBP is {\it extinguished} and in 
the second case we will say that extinction occurs at a finite time and that the MCSBP {\it becomes extinct}. 
Then, we wish to find conditions bearing on the mechanism $\varphi$ allowing us to determine if extinction holds 
at a finite time or not. The probability for ${\rm Z}$ to be extinguished at infinity under $\p_{\rm r}$, 
${\rm r}\in\mathbb{R}_+^d$ will be denoted as follows:
\begin{equation}\label{5348}
\bar{q}_{\rm r}:=\p_{\rm r}\left(\lim\limits_{t\rightarrow\infty}{\rm Z}_t=\mathbf{0}\,
\text{ and }\,{\rm Z}_t>\mathbf{0},\;\mbox{for all $t>0$}\right).
\end{equation}
Let us briefly recall the case $d=1$, from \cite{gr}. Let $\varphi$ be the branching mechanism of the continuous 
state branching process $Z$. If the following integral condition, called Grey's condition, is satisfied,
\[(G_\varphi)\qquad\int^\infty\frac {ds}{\varphi(s)}<\infty,\]
then extinction can only occur at a finite time. Moreover, under $(G_\varphi)$, if $Z$ starts from $r\ge0$, then 
the probability of extinction is $e^{-r\phi(0)}$, where $\phi$ is the inverse of $\varphi$.
If $(G_\varphi)$ is not satisfied, then extinction of $Z$ can only occur at infinity.\\

Let us now consider the general case $d\ge1$. To this aim, we denote by $\tilde{\varphi}_{i}$ the 
Laplace exponent of the spectrally positive L\'evy process $X^{i,i}$, that is 
\begin{equation}\label{2526}
\tilde{\varphi}_{i}(s)=\varphi_i(s{\rm e}_i),\quad s\ge0.
\end{equation}
Note that since $(H_\varphi)$ is satisfied, none of the processes $X^{i,i}$ is a subordinator, so that
$\lim\limits_{s\rightarrow\infty}\tilde{\varphi}_{i}(s)=\infty$, for all $i\in[d]$.
Then let us introduce the following condition bearing on this Laplace exponent,
\[(G^{(i)}_{\varphi}) \qquad\qquad \int^{\infty} \dfrac{ds}{\tilde{\varphi}_{i}(s)}<\infty\,.\ \]
We are now able to state our main result. It provides an extension for MCSBP of Grey's condition.
Recall the definition of $q_{\rm r}$ and $\bar{q}_{\rm r}$ in (\ref{5347}) and (\ref{5348}).

\begin{theorem}\label{4588}
Let ${\rm Z}$ be a MCSBP such that none of the processes $X^{i,j}$ for $i,j\in[d]$, 
$i\neq j$ is a compound Poisson process. 
\begin{itemize}
\item[$1.$] Assume that $(G_{\varphi}^{(i)})$ is satisfied for all $i\in[d]$. Then ${\rm Z}$ can only 
become extinct at a finite time, that is $\bar{q}_{\rm r}=0$ for all ${\rm r}\in\mathbb{R}^d_+$. Moreover,   
\begin{equation}\label{3426}
q_{\rm r}=e^{-\langle {\rm r},\phi(\mathbf{0})\rangle},
\end{equation}
for all ${\rm r}\in\mathbb{R}^d_+$, and in this case, 
$q_{\rm r}=\p_{\rm r}\left(\lim\limits_{t\rightarrow\infty}{\rm Z}_t=\mathbf{0}\right)$.
\item[$2.$] Assume that $(G_{\varphi}^{(i)})$ is not satisfied for some $i\in[d]$. Then, ${\rm Z}$ cannot 
become extinct at a finite time i.e. $q_{\rm r}=0$ for all ${\rm r}\in\mathbb{R}_+^d\setminus\{\mathbf{0}\}$. 
Moreover,
\begin{equation}\label{3427}
\bar{q}_{\rm r}=e^{-\langle {\rm r},\phi(\mathbf{0})\rangle},
\end{equation}
for all ${\rm r}\in\mathbb{R}^d_+\setminus\{\mathbf{0}\}$, and in this case, 
$\bar{q}_{\rm r}=\p_{\rm r}\left(\lim\limits_{t\rightarrow\infty}{\rm Z}_t=\mathbf{0}\right)$.
\end{itemize}
\end{theorem}
\noindent Note that as far as we know, part 2.~of Theorem \ref{4588} when $d=1$ is never mentioned in the 
literature.\\

As mentioned in the introduction, the fact that extinction at infinity is only determined by the law of the
diagonal L\'evy processes $X^{i,i}$ may seem quite surprising at first glance. 
The case of neutral MCSBP's for which part 1.~of Theorem \ref{4588} can be verified directly, gives us a better 
intuition of this result. Recall that a MCSBP is said to be neutral if  the law of $\sum_{i=1}^dX^{i,j}$ does 
not depend on $j$. It is plain that in this case, $Z^{(1)}+\dots+Z^{(d)}$ is a continuous state branching 
process whose branching mechanism is $d\hat{\varphi}$, where $\hat{\varphi}$ is the Laplace exponent of 
$\sum_{i=1}^dX^{i,j}$. Therefore, from Grey's result, ${\rm Z}$ becomes extinct at a finite time with positive 
probability if and only if 
\begin{equation}\label{6566}
\int^\infty\frac{ds}{\hat{\varphi}(s)}<\infty.
\end{equation}
Let us denote by $\varphi_{ij}$ the Laplace exponent of the L\'evy process 
$X^{i,j}$, that is 
\begin{equation}\label{1736}
\varphi_{ij}(s)=\varphi_j(s{\rm e}_i),\quad s\ge0,
\end{equation}
and note that $\tilde{\varphi}_i=\varphi_{ii}$ according to our previous notation. Note also that 
$\hat{\varphi}(s)=\varphi_j(s,s,\dots,s)$, for all $s\ge0$ and $j\in[d]$. Then from Lemma \ref{7256} and 
inequality (\ref{5558}),
\begin{equation}\label{5031}
\sum_{i=1}^d\varphi_{ij}(s)\le \hat{\varphi}(s)\le \tilde{\varphi}_j(s)\,,\;\;\;j\in[d],\;s\ge0.
\end{equation}
Therefore if $(\ref{6566})$ holds, then $(G_{\varphi}^{(i)})$ holds for all $j\in[d]$. Now assume that
$(G_{\varphi}^{(i)})$ holds for all $j\in[d]$. Then from the asymptotic behaviour of Laplace exponents of 
spectrally positive L\'evy processes, see Chap.~VII and in \cite{be}, 
$\lim_{s\rightarrow\infty}\tilde{\varphi}_j(s)/s=\infty$ and since $(\sum_{i\neq j}\varphi_{ij}(s))/s$ is 
bounded by a constant when $s$ is large, $\sum_{i=1}^d\varphi_{ij}(s)$ is equivalent to $\tilde{\varphi}_j(s)$ 
when $s$ tends to infinity. This means that $\int^\infty\frac{ds}{\sum_{i=1}^d\varphi_{ij}(s)}<\infty$, 
so that $(\ref{6566})$ holds. Our proof in 
the general case is based on similar arguments. In order to prove part 1.~of Theorem \ref{4588}, we will 
actually construct a function $\hat{\varphi}$ (no necessarily being a Laplace exponent) which minimizes all 
Laplace exponents $\tilde{\varphi}_j$ and such that condition (\ref{6566}) implies part $(i)$ in Theorem 
\ref{5003}. The general idea is that the distinction between extinction at a finite time and extinction at 
infinity depends only on the asymptotic behaviour of $\varphi_i(\lambda)$, $i\in[d]$, when 
$\lambda\rightarrow\infty$ and this is the same as the asymptotic behaviour of $\tilde{\varphi}_i(\lambda_i)$, 
$i\in[d]$, when $\lambda_i\rightarrow\infty$. Let us emphasize that conditions for conservativeness  
certainly depends on the law of the subordinators $X^{i,j}$. Again, one can easily convince ourself of this 
fact by looking at the neutral case.\\

We end this section by saying a few words about the exclusion of compound Poisson processes.
If for instance, type $1$ is such that for all $j\neq 1$, $X^{1,j}$ is either 
identically equal to 0 or is a compound Poisson process, then there is an a.s.~positive random 
time $\tau$ such that $X^{1,j}_t=0$, for all $j\neq 1$ and $t\in[0,\tau]$. Hence we derive from (\ref{Lamperti}) 
that if ${\rm r}$ is such that $r_1=0$, then $Z^{(1)}_t=0$, for all $t\in[0,\gamma]$ for some a.s.~positive random 
time $\gamma$. If moreover conditions $(G_{\varphi}^{(j)})$ are satisfied for all $j\neq 1$, then from part 
$(ii)$ of the theorem, the process ${\rm Z}$ restricted to types $2,3,\dots,d$ can reach ${\bf 0}$ with positive
probability. Moreover, it is reasonable to think that this hitting time of ${\bf 0}$ can occur in the time 
interval $[0,\gamma]$ with positive probability. Therefore, the probability of extinction at a finite 
time of ${\rm Z}$ is positive in this case. Actually, whenever there exist processes $X^{i,j}$, $i\neq j$ that are 
compound Poisson processes, it is possible that more than one condition $(G_{\varphi}^{(i)})$ is required for the 
process not to become extinct at a finite time. A general result including compound Poisson processes could be
proved, but we do not think this case is much relevant.  

\section{Proof of Theorems \ref{1005} and \ref{5003}}\label{proofs.ext}

\noindent {\it Proof of Theorem $\ref{1005}$}:
The Lamperti representation (\ref{Lamperti}) yields the following path by path 
relationship between ${\rm Z}$ starting at ${\rm Z}_0={\rm r}$ and the first passage times of 
its underlying spaLf,   
\begin{equation}\label{8556}
\left\{\lim\limits_{t\rightarrow\infty}{\rm Z}_t=\mathbf{0}\right\}\subset
\left\{\mathbf{T}_{{\rm r}'}\in\mathbb{R}_+^d\right\},\;\;\;
\mbox{for all ${\rm r}'$ such that $0\le r_i'<r_i$, for all $i\in[d]$.}
\end{equation}
Indeed, it follows directly from equation (\ref{Lamperti}) that for ${\rm Z}$ to be as close to $\mathbf{0}$ 
as possible, the spaLf ${\bf X}$ has to reach all possible levels ${\rm r}'$ such that $0\le r_i'<r_i$, 
for all $i\in[d]$. Together with (\ref{2679}) inclusion (\ref{8556}) implies that 
\[\p_{\rm r}\left(\lim\limits_{t\rightarrow\infty}{\rm Z}_t=\mathbf{0}\right)\le e^{-\langle {\rm r}',
\phi(\mathbf{0})\rangle},\;\;\;\mbox{for all ${\rm r}'$ such that $0\le r_i'<r_i$, for all $i\in[d]$,}\]
which gives one inequality in (\ref{5127}) by continuity.

In order to prove the other inequality, let us set 
$\int_0^\infty {\rm Z}_s\,ds=\left(\int_0^\infty Z_s^{(i)}\,ds\right)_{i\in[d]}$ and 
note that from (\ref{Lamperti}), on the set $\left\{\int_0^\infty {\rm Z}_s\,ds\in\mathbb{R}_+^d\right\}$, 
the process ${\rm Z}$ converges $\p_{\rm r}$-almost surely and its limit is $\mathbf{0}$. This implies the 
inclusion,
\[\left\{\int_0^\infty {\rm Z}_s\,ds\in\mathbb{R}_+^d\right\}
\subset\left\{\lim_{t\rightarrow\infty}{\rm Z}_t=\mathbf{0}\right\},\;\;\;\mbox{$\p_{\rm r}$-a.s.}\]
Then recall from Proposition \ref{8260} that, 
${\bf T}_{\rm r}=\int_0^\infty {\rm Z}_s\,ds$,
$\p_{\rm r}$-a.s. on the set $\{{\rm Z}_t\in\mathbb{R}^d_{+},\,t\ge0\}$ (which is supposed to be of probability 
1 here), for all ${\rm r}\in\mathbb{R}_+^d$, so that from (\ref{2679}) and the above inclusion,
\[\p({\bf T}_{\rm r}\in\mathbb{R}_+^d)=e^{-\langle {\rm r},\phi(\mathbf{0})\rangle}\le
\p_{\rm r}\left(\lim_{t\rightarrow\infty}{\rm Z}_t=\mathbf{0}\right),\]
which achieves our proof. $\quad\Box$\\

\noindent The next two lemmas will be needed for the proof of Theorem \ref{5003}.

\begin{lemma}\label{u.D}
For all $\lambda\in D_{\varphi}$ and $t\ge0$,
${\rm u}_{t}(\lambda)\in D_{\varphi}$ and for all $i\in[d]$, the mapping $t\mapsto u^{(i)}_{t}(\lambda)$ 
is decreasing. Moreover for all $i\in[d]$, the mapping $t\mapsto u^{(i)}_{t}(\infty)$ is decreasing.
\end{lemma}
\begin{proof}
Let $\lambda\in D_{\varphi}$. Then for all $i\in[d]$,   
$\dfrac{\partial}{\partial t}u^{(i)}_{t}(\lambda)|_{t=0}=-\varphi_{i}(\lambda)<0$ so that for all $h$ small 
enough, $u_{h}^{(i)}(\lambda)<\lambda_i=u_{0}^{(i)}(\lambda)$. That is 
$[{\rm u}_{h}(\lambda),\lambda]:=\prod_{i\in[d]}[u_h^{(i)}(\lambda),\lambda_i]\neq\emptyset$.
Now let us fix $i\in[d]$. Then thanks to the theorem of finite increments, for all $h\geq 0$ small enough, 
there exists $\lambda'\in [{\rm u}_{h}(\lambda),\lambda]$ such that
\[u^{(i)}_{t}({\rm u}_{h}(\lambda))-u^{(i)}_{t}(\lambda)=u_{t+h}^{(i)}(\lambda)-u^{(i)}_{t}(\lambda)
=\sum\limits_{j=1}^{d} \left(u^{(j)}_{h}(\lambda)-\lambda_{j}\right) \dfrac{\partial}{\partial 
\lambda_{j}}u^{(i)}_{t}(\lambda')\,,\ \]
where we have used the semigroup property of ${\rm u}_t(\lambda)$ in the first equality. Then making $h$ tend 
to zero, we obtain
\begin{eqnarray*}
\lim\limits_{h\rightarrow 0} \dfrac{u_{t+h}^{(i)}(\lambda)-u^{(i)}_{t}(\lambda)}{u^{(i)}_{h}
(\lambda)-\lambda_{i}}&=& \dfrac{\partial}{\partial \lambda_{i}}u^{(i)}_{t}(\lambda)+\\
&&\sum\limits_{j\neq i}\dfrac{\partial}{\partial t}u^{(j)}_{t}(\lambda)|_{t=0}
\left(\dfrac{\partial}{\partial t}u^{(i)}_{t}(\lambda)|_{t=0}\right)^{-1} \dfrac{\partial}
{\partial \lambda_{j}}u^{(i)}_{t}(\lambda)\\
&=&\dfrac{\partial}{\partial \lambda_{i}}u^{(i)}_{t}(\lambda)+\sum\limits_{j\neq i}
\dfrac{\varphi_j(\lambda)}{\varphi_i(\lambda)}\dfrac{\partial}
{\partial \lambda_{j}}u^{(i)}_{t}(\lambda).
\end{eqnarray*}
This last quantity is positive. Indeed, it follows from Lemma 3.4 in \cite{blp} and part $(ii)$ of Lemma A.1  
in \cite{bp}, that $\e_{e_i}(Z_t^{(j)})>0$, for all $t\ge0$ and all $i,j\in[d]$. (In \cite{blp} it is 
actually assumed that $X^{i,j}$, $i,j\in[d]$ are integrable but our result can easily be extended to
the general case.) Therefore, for every $i,j\in[d]$, $\lambda_{j}\mapsto u_{t}^{(i)}(\lambda)$ is increasing. 
Moreover, since $\lambda\in D_\varphi$, $\varphi_j(\lambda)>0$, for all $j\in[d]$. This yields,
\begin{align*}
\lim\limits_{h\rightarrow 0} \dfrac{u^{(i)}_{t+h}(\lambda)-u^{(i)}_{t}
	(\lambda)}{u^{(i)}_{h}(\lambda)-\lambda_{i}}
	&=\dfrac{\partial}{\partial t}u^{(i)}_{t}(\lambda)
	\left( \dfrac{\partial}{\partial t}u^{(i)}_{t}(\lambda)|_{t=0}\right)^{-1}>0.
\end{align*}
Thus $\dfrac{\partial}{\partial t}u^{(i)}_{t}(\lambda)$ has the same sign as $\dfrac{\partial}{\partial t}
u^{(i)}_{t}(\lambda)|_{t=0}=-\varphi_{i}(\lambda)<0$ since $\lambda\in D_{\varphi}$. 
In conclusion, it follows from the differential system $(S_\varphi)$ satisfied by ${\rm u}$ and $\varphi$, that 
for all $\lambda\in D_{\varphi}$, $t\geq 0$, and $i\in[d]$
\[\dfrac{\partial}{\partial t}u^{(i)}_{t}(\lambda)
	=-\varphi_{i}({\rm u}_{t}(\lambda))<0,\]
that is
${\rm u}_{t}(\lambda)\in D_{\varphi}$ and for all $i\in[d]$, $t\mapsto u^{(i)}_{t}(\lambda)$ is decreasing.

The last assertion is a consequence of the previous one, but it is also straightforward from (\ref{2782}).\\
\end{proof}

\begin{lemma}\label{asympu}
Assume that none of the processes $X^{i,j}$ for $i,j\in[d]$ such that $i\neq j$ is a compound Poisson process.  
Then either for all $t>0$ and $i\in[d]$, $u^{(i)}_{t}(\infty)=\infty$, 
or for all $t>0$ and $i\in[d]$, $u^{(i)}_{t}(\infty)<\infty$.
\end{lemma}
\begin{proof}
Assume that there exists $t>0$ such that $u^{(i)}_{t}(\infty)=\infty$ for some $i\in[d]$. Let 
$j$ such that $j\neq i$ and $X^{i,j}$ is not identically equal to 0. Such an index exists since ${\bf X}$ 
is irreducible. Assume moreover that $u^{(j)}_{t}(\infty)<\infty$. Since $X^{i,j}$ is not a compound Poisson 
process, $\varphi_j$ tends to $-\infty$ on the coordinate $i$. Therefore, since $u^{(i)}_{t}(\infty)=\infty$, 
we have $\varphi_{j}({\rm u}_{t}(\infty))=-\infty$ and this 
contradicts the fact that ${\rm u}_{t}(\infty)\in \overline{D}_\varphi$ following from Lemma \ref{u.D}. In other 
words, for all $j\in[d]$ such that $j\neq i$ and $X^{i,j}$ is not identically equal to 0, 
$u^{(j)}_{t}(\infty)=\infty$. By irreducibility, we deduce that for all $j\in[d]$, $u^{(j)}_{t}(\infty)=\infty$.
We have proved that, if there exists $t>0$ and $i\in[d]$ such that $u^{(i)}_{t}(\infty)<\infty$, then for 
all $j\in[d]$, $u^{(j)}_{t}(\infty)<\infty$. Moreover from Lemma \ref{u.D}, for all $s\geq t$ and $j\in[d]$, 
$u^{(j)}_{s}(\infty)<\infty$.

Now let $\tau^{(i)}=\inf\{s\geq 0: u^{(i)}_{s}(\infty)<\infty\}$. Thanks to previous arguments, $\tau=\tau^{(i)}$ 
does not depend on $i\in[d]$. Assume $\tau\in(0,\infty)$.
For $0<s<\tau$ and $\tau-s<t<\tau$, for all $i\in[d]$ and $\lambda\in D_{\varphi}$,
\[ u^{(i)}_{t+s}(\lambda) 
	= u^{(i)}_{t}({\rm u}_{s}(\lambda))
	\underset{\lambda\rightarrow +\infty}{\rightarrow} u^{(i)}_{t}(\infty) = \infty\,.\ \]	
We obtain a contradiction since $u^{(i)}_{t+s}(\infty)<\infty$.
\end{proof}

\noindent {\it Proof of Theorem $\ref{5003}$}: It has already been proved in Lemma \ref{asympu} that 
either for all $t\geq 0$ and $i\in[d]$, $u^{(i)}_{t}(\infty)=\infty$ or for all $t>0$ and $i\in[d]$, 
$u^{(i)}_{t}(\infty)<\infty$. Now recall from (\ref{2782}) that 
$\p_{\rm r}({\rm Z}_t=\mathbf{0})=e^{-\langle {\rm r},{\rm u}_t(\infty)\rangle}$. Hence if
for all $t\geq 0$ and $i\in[d]$, $u^{(i)}_{t}(\infty)=\infty$, then for all $t\geq 0$ and 
${\rm r}\in\mathbb{R}^{d}_{+}\setminus\{\mathbf{0}\}$, $\p_{\rm r}({\rm Z}_t=\mathbf{0})=0$, so that for 
all ${\rm r}\in\mathbb{R}^{d}_{+}\setminus\{\mathbf{0}\}$, 
$q_{\rm r}=\lim\limits_{t\rightarrow\infty}\p_{\rm r}({\rm Z}_t=\mathbf{0})=0$.

Again from (\ref{2782}), if for all $t\geq 0$ and $i\in[d]$, $u^{(i)}_{t}(\infty)<\infty$, then for all 
$t\geq 0$ and ${\rm r}$, 
$\p_{\rm r}({\rm Z}_t=\mathbf{0})=e^{-\langle{\rm r},{\rm u}_{t}(\infty)\rangle}>0$ and since all coordinates 
of $t\mapsto {\rm u}_t(\infty)$ are decreasing, $q_{\rm r}\in(0,1]$. 
Let us prove that in this case, $q_{\rm r}=e^{-\langle {\rm r},\phi({\bf 0})\rangle}$. From (\ref{5627}) 
and (\ref{2782}) this equality is equivalent to 
\begin{equation}\label{0801}
{\rm k}=(k^{(1)},\dots,k^{(d)}):=\lim_{s\rightarrow\infty}{\rm u}_{s}(\infty)=\phi({\bf 0}). 
\end{equation}
In order to prove (\ref{0801}), we will use the semi-group property. That is for all $i\in[d]$, $s,t\geq 0$ 
and $\lambda\in\mathbb{R}^{d}_{+}$,
\begin{equation}\label{limsemigp}
u_{t+s}^{(i)}(\lambda)=u_{t}^{(i)}({\rm u}_{s}(\lambda)).
\end{equation}
Recall that for all $s>0$, ${\rm u}_{s}(\infty)\in\mathbb{R}^{d}_{+}$. Then using  
(\ref{limsemigp}) and making $\lambda$ tend to infinity, we obtain by the 
continuity of $u_{t}^{(i)}$,
\begin{equation}\label{5108}
u_{t+s}^{(i)}(\infty)=u_{t}^{(i)}({\rm u}_{s}(\infty))\,.
\end{equation}
Recall from (\ref{0801}) the definition of ${\rm k}$. Since the mappings $s\mapsto u^{(i)}_{s}(\infty)$, 
$i\in[d]$ are decreasing, see (\ref{2782}), we have ${\rm k}\in\mathbb{R}_+^d$, so that
from (\ref{5108}) and by continuity of $u_{t}^{(i)}$ again, we obtain when $s$ tends to $\infty$ that 
for all $i\in[d]$,
\begin{equation}\label{5208}
k^{(i)}=u_{t}^{(i)}({\rm k})\,.
\end{equation}
Therefore $t\mapsto {\rm u}_{t}({\rm k})$ is constant, so that from 
$(S_{\varphi})$, $\varphi({\rm k})={\bf 0}$ and hence from part 3.~of Theorem \ref{propTr},
${\rm k}\in\{\mathbf{0},\phi(\mathbf{0})\}$.
Moreover, thanks to Lemma $\ref{u.D}$, for all $\lambda\in D_{\varphi}$ and $t\geq 0$, 
${\rm u}_{t}(\lambda)\in D_{\varphi}$, thus ${\rm k}\in \overline{D}_{\varphi}$.
Then it is proved at the beginning of the proof of Theorem 3.4 in \cite{cma1} that $\phi({\bf 0})$ is the only 
solution of the equation $\varphi(\lambda)={\bf 0}$ in $\overline{D}_\varphi$. This, shows that 
${\rm k}=\phi(\mathbf{0})$ and we conclude that $q_{\rm r}=e^{-\langle {\rm r},\phi({\bf 0})\rangle}$. 
$\quad\Box$

\section{Proof of Theorem \ref{4588}}\label{proofs.Grey.M}

For $i\in[d]$, let us call \emph{the $i^{th}$ diagonal branching process}, the continuous state branching process
${\tilde{Z}^{(i)}}$ which is solution of the one-dimensional Lamperti representation,
\begin{equation}\label{Lampertidiagonale}
\tilde{Z}^{(i)}_{t} = r_{i} + X^{i,i}\left(\int\limits_{0}^{t} \tilde{Z}^{(i)}_{s}ds\right), \;\;\; t\geq 0.
\end{equation}
Its branching mechanism corresponds to the Laplace exponent $\tilde{\varphi}_{i}$ of $X^{i,i}$ whose definition 
has been given in (\ref{2526}), see also (\ref{1736}) below.
Then the Laplace exponent $\tilde{u}^{(i)}$ of $\tilde{Z}^{(i)}$ satisfies for all $\alpha\geq 0$,
\[ \left\lbrace\begin{array}{ll}
	\dfrac{\partial}{\partial t} \tilde{u}^{(i)}_{t}(\alpha) &= -\tilde{\varphi}_{i}(\tilde{u}^{(i)}_{t}(\alpha))
	\\ \tilde{u}^{(i)}_{0}(\alpha) &= \alpha.
	\end{array}\right. \]
Extinction at a finite time of the $i^{th}$ diagonal branching process is characterized through Grey's condition 
for $\tilde{\varphi}_{i}$, see $(G_{\varphi}^{(i)})$ right after (\ref{2526}).\\

Let us start with the proof of part 1.~of Theorem \ref{4588}. Let $\varphi$ be the Laplace exponent of some spaLf 
satisfying  $(G^{(i)}_{\varphi})$, for all $i\in[d]$. We will show that $u^{(i)}_t(\infty)<\infty$, for all 
$t\ge0$ and $i\in[d]$ and use Theorem \ref{5003}. For this purpose, we need the following series of Lemmas.

We first show that in the case $d=1$, for some functions that are more general than Laplace exponents of  
spectrally positive L\'evy processes, the system $(S_\varphi)$ defined in Subsection \ref{9522} still 
admits a unique solution with nice properties.

\begin{lemma}\label{9525}
Let $s_0\ge0$ and $f:[s_0,\infty)\rightarrow[0,\infty)$ be some continuous, increasing function such that 
$f(s_0)=0$ and $\int_{s_0}ds/f(s)=\infty$. Fix $\delta>s_0$ and let 
\[F(x):=\int_\delta^x\frac{ds}{f(s)}\,,\;\;\;x\in(s_0,\infty).\]
Set $F(\infty):=\lim\limits_{x\rightarrow\infty}F(x)$. Then $F:(s_0,\infty)\rightarrow(-\infty,F(\infty))$ is a 
bijection. Let us denote by 
$F^{-1}:(-\infty,F(\infty))\rightarrow(s_0,\infty)$ its inverse. Then for all $\lambda>s_0$, $v_t(\lambda)=F^{-1}(F(\lambda)-t)$, $t\ge0$ is the unique solution of the differential equation
\begin{equation}\label{5636}
\left\{\begin{array}{ll}
&\frac{\partial}{\partial_t}v_t(\lambda)=-f(v_t(\lambda))\,,\;\;\;t>0,\\
&u_0(\lambda)=\lambda\,.
\end{array}\right.
\end{equation}
For every $\lambda>s_0$, the function $t\mapsto v_t(\lambda)$ is decreasing and $\lim\limits_{t\rightarrow\infty}v_t(\lambda)=s_0$. For every $t\ge0$, the function $\lambda\mapsto v_t(\lambda)$ is increasing.
Moreover, $v_t(\infty):=\lim\limits_{\lambda\rightarrow\infty}v_t(\lambda)<\infty$, for all 
$t>0$ if and only if $\int_{\delta}^\infty ds/f(s)<\infty$.
\end{lemma}
\begin{proof} It is plain that $F:(s_0,\infty)\rightarrow(-\infty,F(\infty))$ is a bijection.
Let $\lambda>s_0$. Then the function $t\mapsto F^{-1}(F(\lambda)-t)$ is clearly differentiable and 
we readily check that it satisfies (\ref{5636}). Uniqueness of the solution is given by integrating the function 
$t\mapsto -\frac{\partial}{\partial_t}v_t(\lambda)/f(v_t(\lambda))$ through an obvious change of variables 
allowing us to write
\begin{equation}\label{9299}
\int_{v_t(\lambda)}^\lambda\frac{ds}{f(s)}=t\,,
\end{equation}
and the expression of $v_t$ follows. The facts that the function $t\mapsto v_t(\lambda)$ is decreasing, that
$\lim\limits_{t\rightarrow\infty}v_t(\lambda)=s_0$ and that the function $\lambda\mapsto v_t(\lambda)$ is 
increasing are straightforward. Finally (\ref{9299}) implies that for every $t>0$, 
$v_t(\infty)<\infty$ if and only if $\int_{\delta}^\infty ds/f(s)<\infty$.
\end{proof}

\noindent  Recall from (\ref{1736}) that $\varphi_{ij}$ is the Laplace exponent of the L\'evy process 
$X^{i,j}$ and that $\varphi_{ii}=\tilde{\varphi}_{i}$ according to our notation.
Then let us consider the following mappings: 
\begin{eqnarray*}
\bar{\varphi}:=\min\{\tilde{\varphi}_{i},\,i\in[d]\}\,,\;\;\bar{\bar{\varphi}}:=\sum_{k\neq l}\varphi_{kl}\,,\;\;
\varphi^*_i(\lambda):=\bar{\varphi}(\lambda_i)+\sum_{j\neq i}\bar{\bar{\varphi}}(\lambda_j).
\end{eqnarray*}
\noindent Let us define the diagonal on $\mathbb{R}_+^d\setminus\{\mathbf{0}\}$ by 
$\Delta=\{\lambda\in\mathbb{R}_+^d\setminus\{\mathbf{0}\}:\lambda_1=\lambda_2=\dots=\lambda_d\}$.

\begin{lemma}\label{8116}
Recall that the assumptions $(G^{(i)}_\varphi)$, $i\in[d]$ are in force. Let $s_0$ be the largest solution of
the equation $\bar{\varphi}(x)+(d-1)\bar{\bar{\varphi}}(x)=0$. Then for all $\lambda\in\Delta$ with 
$\lambda_1>s_0$, the differential system 
\[ (S_{\varphi^*})\;\;\qquad\;
\left\lbrace
		\begin{array}{ll}
			\dfrac{\partial}{\partial t} v^{(i)}_{t}(\lambda)  & = -\varphi^*_{i}({\rm v}_{t}(\lambda))
			\\[0.3cm] v^{(i)}_{0}(\lambda)&=\lambda_{i}
		\end{array}
\;\;\;,\;\;i\in[d]\,,\;\;t\ge0
	\right.\]
admits as a solution the mapping $t\mapsto {\rm v}_t(\lambda)$ given by 
\[v_t^{(1)}(\lambda)=\dots=v_t^{(d)}(\lambda)=F^{-1}(F(\lambda)-t),\] 
for all $t\ge0$, where $F$ is given as in Lemma $\ref{9525}$, with 
$f(x):=\bar{\varphi}(x)+(d-1)\bar{\bar{\varphi}}(x)$, $x\ge s_0$.
\end{lemma}
\begin{proof} Let us check that $f$ satisfies the conditions of Lemma \ref{9525}. First note that
$\bar{\varphi}+(d-1)\bar{\bar{\varphi}}$ is clearly continuous on $[0,\infty)$. Set 
$f_i(x):=\tilde{\varphi}_i(x)+(d-1)\bar{\bar{\varphi}}(x)$, $x\ge0$. Then $f_i$ is the characteristic exponent 
of a spectrally positive L\'evy process which is not a subordinator under our assumptions. Therefore, 
$f_i$ is convex with $\lim_{x\rightarrow\infty}f_i(x)=\infty$, and the equation $f_i(x)=0$ has at most two 
roots. Let $s_i$ be the largest of these roots. Then $f_i$ is increasing on $[s_i,\infty)$. Moreover since 
$\bar{\varphi}+(d-1)\bar{\bar{\varphi}}=\min\{f_i,i\in[d]\}$, the point $s_0=\max\{s_i,i\in[d]\}$ 
is the largest root of the equation $\bar{\varphi}(x)+(d-1)\bar{\bar{\varphi}}(x)=0$. Therefore, the function
$f:[s_0,\infty)\rightarrow[0,\infty)$ defined by $f(x):=\bar{\varphi}(x)+(d-1)\bar{\bar{\varphi}}(x)$, $x\ge s_0$ 
is continuous, increasing and satisfies $f(s_0)=0$.

Now let $I\subset[d]$ be the set of indices such that $s_0=s_i$, for all $i\in I$. Then there is $\varepsilon>0$ 
such that $f(x)=\min\{f_i(x),i\in I\}$ for all $x\in[s_0,s_0+\varepsilon]$. Moreover, for all $i\in I$, 
\[\infty=\int_{s_0}^{s_0+\varepsilon}\frac{dx}{f_i(x)}\le\int_{s_0}^{s_0+\varepsilon}\frac{dx}{f(x)}\,,\]
where the equality follows from the fact that $0\le f_i'(s_0)<\infty$.

The result is then a consequence of Lemma \ref{9525} and the fact that for all $\lambda\in\Delta$ such that 
$\lambda_1>s_0$ and for all $i,j\in[d]$, $\varphi^*_i(\lambda)=\varphi^*_j(\lambda)=f(\lambda_1)$. 
(Note also that $t\mapsto{\rm v}_t(\lambda)$ is the only solution such that 
$v_t^{(1)}(\lambda)=\dots=v_t^{(d)}(\lambda)$.) 
\end{proof}

\begin{lemma}\label{7256} 
Assume that $d\ge2$. Then for all $j\in[d]$ and $\lambda\in\mathbb{R}_+^d$,
\[\sum_{i=1}^d\varphi_{ij}(\lambda_i)\le\varphi_j(\lambda).\]
Moreover, for all $j\in[d]$ there exists $i\neq j$ such that $\varphi_{ij}\neq 0$. As a consequence, 
for all $j\in[d]$ and $\lambda\in(\mathbb{R}_+\setminus\{0\})^d$,
\[\varphi_j^*(\lambda)<\varphi_j(\lambda).\]
\end{lemma}
\begin{proof}  By definition,
\[\sum_{i=1}^d\varphi_{ij}(\lambda_i)=-\sum_{i=1}^da_{i,j}\lambda_i+\frac12q_j\lambda_j^2+
\int_{\mathbb{R}_+^d}\langle\lambda, {\rm x}\rangle1_{\{|{\rm x}|<1\}}+\sum_{i=1}^d(e^{-\lambda_ix_i}-1)\,
\pi_j(d{\rm x})\]
and the first inequality is a consequence of the following one, 
$\sum\limits_{i=1}^d(e^{-\lambda_i x_i}-1)\le e^{-\sum\limits_{i=1}^d\lambda_ix_i}-1$ which is valid for
all $\lambda,{\rm x}\in\mathbb{R}_+^d$. The second assertion follows from irreducibility. The last 
assertion is a consequence of the first inequality and the inequality 
$\varphi_j^*(\lambda)<\sum\limits_{i=1}^d\varphi_{ij}(\lambda_i)$, for all $j\in[d]$ and 
$\lambda\in(\mathbb{R}_+\setminus\{0\})^d$. Indeed, 
\begin{eqnarray*}
\varphi_j^*(\lambda)&=&\bar{\varphi}(\lambda_j)+\sum_{i\neq j}\sum_{k\neq l}\varphi_{kl}(\lambda_i)\\
&=&\bar{\varphi}(\lambda_j)+\sum_{i\neq j}\varphi_{ij}(\lambda_i)+
\sum_{i\neq j}\sum_{k\neq l,(k,l)\neq (i,j)}\varphi_{kl}(\lambda_i)\\
&\le&\sum_{i=1}^d\varphi_{ij}(\lambda_i)+\sum_{i\neq j}\sum_{k\neq l,(k,l)\neq (i,j)}\varphi_{kl}(\lambda_i).
\end{eqnarray*}
Moreover, for the second term, we have 
\[\sum_{i\neq j}\sum_{k\neq l,(k,l)\neq (i,j)}\varphi_{kl}(\lambda_i)\le\sum_{i\neq j}
\sum_{k\neq i}\varphi_{ki}(\lambda_i)<0,\]
where the last inequality follows from the fact that for all $i\in[d]$ there exists 
$k\neq i$ such that $\varphi_{ki}\neq 0$.
\end{proof}

\begin{corollary}\label{comparaison.infty}
Assume that $d\ge2$. Then for all $\lambda\in\Delta$ and $t\geq 0$, 
${\rm u}_{t}(\lambda)\leq {\rm v}_{t}(\lambda)$.
\end{corollary}
\begin{proof}
Fix $\lambda\in\Delta$ and recall that for all $t\ge0$, ${\rm v}_t(\lambda)\in\Delta$. 
Then from Lemma \ref{7256},
\begin{align*}
{\rm u}_{0}(\lambda) &= {\rm v}_{0}(\lambda)=\lambda
\\ \text{ and } \;\;\;
\dfrac{\partial}{\partial t} u^{(i)}_{t}(\lambda)|_{t=0}
	=-\varphi_{i}(\lambda)
	&< -\varphi^*_{i}(\lambda)
	=\dfrac{\partial}{\partial t} v^{(i)}_{t}(\lambda)|_{t=0}, \; i\in[d].
\end{align*}
Thanks to Taylor's formula, for $t>0$ small enough, $u^{(i)}_{t}(\lambda) < v^{(i)}_{t}(\lambda)$ for all 
$i\in[d]$.

Let $\tau^{(i)}=\inf\{t>0: u^{(i)}_{t}(\lambda)>v^{(i)}_{t}(\lambda)\}>0$ and 
$\tau=\min\limits_{i\in[d]}\tau^{(i)}>0$. Assume $\tau< \infty$. By definition of $\tau$ and by continuity, 
there exists at least one $i\in[d]$ such that $u^{(i)}_{\tau}(\lambda)=v^{(i)}_{\tau}(\lambda)$ and for all 
$j\neq i$, $u^{(j)}_{\tau}(\lambda)\leq v^{(j)}_{\tau}(\lambda)$.
Recall from Lemma \ref{7256} that for all $\lambda\in(\mathbb{R}_+\setminus\{0\})^d$ and $j\in[d]$, 
$\varphi^*_{j}(\lambda)< \varphi_{j}(\lambda)$ and also that if $k\neq j$, then 
$\lambda_{k}\mapsto \varphi^*_{j}(\lambda)$ is non increasing. Note also that, as justified in the proof of 
Lemma \ref{u.D}, $\e_{\rm r}(Z^{(i)}_\tau)>0$, for all ${\rm r}\in\mathbb{R}^{d}_{+}\setminus\{{\bf 0}\}$ 
and $i\in[d]$ and hence $u^{(i)}_{\tau}(\lambda)>0$, for all $i\in[d]$. 
Let $I=\{i\in[d] : u^{(i)}_{\tau}(\lambda)=v^{(i)}_{\tau}(\lambda)\}$, then for all $i\in I$,
\[\dfrac{\partial}{\partial t} u^{(i)}_{t}(\lambda)|_{t=\tau}
	=-\varphi_{i}({\rm u}_{\tau}(\lambda))
	< -\varphi^*_{i}({\rm u}_{\tau}(\lambda))
	\leq -\varphi^*_{i}({\rm v}_{\tau}(\lambda))
	=\dfrac{\partial}{\partial t} v^{(i)}_{t}(\lambda)|_{t=\tau}\,.\ \] 
This implies that for $t>0$ small enough, $u^{(i)}_{\tau+t}(\lambda) < v^{(i)}_{\tau+t}(\lambda)$ for all 
$i\in I$ and this refutes the assumption $\tau<\infty$. We conclude that for all $\lambda\in \Delta$, 
$t\geq 0$ and $i\in[d]$, $u^{(i)}_{t}(\lambda)\leq v^{(i)}_{t}(\lambda)$.
\end{proof}

We are now able to end the proof of part 1.~of Theorem $\ref{4588}$. Let us first assume that $d\ge2$.
Since $(G^{(i)}_{\varphi})$ are satisfied for all $i\in[d]$, we can check that 
\begin{equation}\label{7144}
\int^\infty \frac{ds}{\bar{\varphi}(s)+(d-1)\bar{\bar{\varphi}}(s)}<\infty.
\end{equation}
Indeed, $\lim\limits_{s\rightarrow\infty}|\bar{\bar{\varphi}}(s)|/s$ is bounded by a constant. Moreover, since 
$(G^{(i)}_{\varphi})$ are satisfied and from the behaviour of Laplace exponents of spectrally positive 
L\'evy processes, see Chap.~VII and in \cite{be}, we have
$\lim\limits_{s\rightarrow\infty}\tilde{\varphi}_{i}(s)/s=\infty$, for all $i\in[d]$. This yields that 
$\lim\limits_{s\rightarrow\infty}\bar{\varphi}(s)/s=\infty$.
Hence the above integral is finite if and only if $\int^\infty\frac{ds}{\bar{\varphi}(s)}<\infty$.
Then note the inequality $\frac{1}{\bar{\varphi}(s)}\le\max_{i\in[d]}\frac{1}{\tilde{\varphi}_{i}(s)}\le
\sum_{i\in[d]}\frac{1}{\tilde{\varphi}_{i}(s)}$, which is valid whenever $s>0$ is such that 
$\tilde{\varphi}_{i}(s)>0$ for all $i\in[d]$. This yields,
\[\int^\infty \frac{ds}{\bar{\varphi}(s)} <\sum_{i\in[d]}
\int^\infty \frac{ds}{\tilde{\varphi}_{i}(s)}<\infty\,.\] 
Now recall the definition of $s_0$ from Lemma \ref{8116} and for $s\ge s_0$, let $t\mapsto{\rm v}_t(s{\bf 1})$, 
be the solution of $(S_{\varphi^*})$ in this lemma. Then still according to this lemma, 
$t\mapsto v_t^{(1)}(s{\bf 1})=\dots=v_t^{(d)}(s{\bf 1})$ is the solution of the differential 
equation (\ref{5636}) in Lemma \ref{9525}, with $f(s)=\bar{\varphi}(s)+(d-1)\bar{\bar{\varphi}}(s)$.
Hence from Lemma \ref{9525} and (\ref{7144}), $\lim\limits_{s\rightarrow\infty}v^{(i)}_{t}(s{\bf 1})<\infty$ 
for all $t>0$ and $i\in[d]$. Then as proved in Lemma \ref{7256}, $\varphi^*(\lambda)<\varphi(\lambda)$, for 
all $\lambda\in(\mathbb{R}_+\setminus\{0\})^d$.
Hence from Corollary \ref{comparaison.infty}, for all $\lambda\in \Delta$ and $i\in[d]$, 
$u^{(i)}_t(\lambda)\le v_t^{(i)}(\lambda)$, so that $u^{(i)}_{t}(\infty)<\infty$, for all $t>0$ and $i\in[d]$.
Part $(ii)$~of Theorem \ref{5003} implies that $q_{\rm r}=e^{-\langle {\rm r},\phi(\mathbf{0})\rangle}$.
But since 
\begin{equation}\label{8226}
\left\{\mbox{${\rm Z}_{t}=\mathbf{0}$ for some $t>0$}\right\}\cup
\left\{\lim\limits_{t\rightarrow\infty}{\rm Z}_t=\mathbf{0} \,\text{ and }\, {\rm Z}_t>\mathbf{0},\;
\mbox{for all $t>0$}\right\}=\left\{\lim\limits_{t\rightarrow\infty}{\rm Z}_t=\mathbf{0}\right\},
\end{equation}
we conclude from Theorem \ref{1005} that ${\rm Z}$ can only 
become extinct at a finite time, that is $\bar{q}_{\rm r}=0$ for all ${\rm r}\in\mathbb{R}^d_+$, and
$q_{\rm r}:=\p_{\rm r}(\mbox{${\rm Z}_{t}=\mathbf{0}$ for some $t>0$})=
\p_{\rm r}(\lim\limits_{t\rightarrow\infty}{\rm Z}_t=\mathbf{0})$. If $d=1$, then we know directly from 
Grey's result that $u_t(\infty)<\infty$, for all $t\ge0$ and the same conclusion follows.\\

Let us now prove part 2.~of Theorem \ref{4588}. We will show that if one of the diagonal branching processes
$\tilde{Z}^{(i)}$ defined at the beginning of this subsection is extinguished at infinity, that is if 
$(G_{\varphi}^{(i)})$ does not hold for some $i\in[d]$, then ${\rm Z}$ also is extinguished at infinity. 
Let $i\in[d]$ and first note that since 
$\varphi_i$ is non increasing in the variables $\lambda_j$, for $j\neq i$, for all $\lambda\in(0,\infty)^d$, 
\begin{equation}\label{5558}
\varphi_{i}(\lambda)\le\varphi_{i}(\lambda_{i}{\rm e}_{i})=\tilde{\varphi}_{i}(\lambda_{i})\,.
\end{equation}
Take $\lambda\in D_{\varphi}$ so that from Lemma \ref{u.D}, ${\rm u}_{s}(\lambda)\in D_\varphi$ 
and in particular, $\varphi_{i}({\rm u}_{s}(\lambda))>0$, for all $s\ge0$.
Then we derive from the differential system $(S_{\varphi})$ that for all $s\ge0$,
$\dfrac{\frac{\partial }{\partial s}u^{(i)}_{s}(\lambda)}{-\varphi_{i}({\rm u}_{s}(\lambda))}=1$, so that
integrating over $[0,t]$ and using the above inequality, we obtain for all $t\geq 0$,
\begin{align*}
t=\int_{0}^{t}
	\dfrac{\frac{\partial }{\partial s}u^{(i)}_{s}(\lambda)}{-\varphi_{i}({\rm u}_{s}(\lambda))}ds
	\ge\int_{0}^{t}
		\dfrac{\frac{\partial }{\partial s}u^{(i)}_{s}(\lambda)}{-\tilde{\varphi_{i}}(u^{(i)}_{s}(\lambda))}ds
	= \int_{u^{(i)}_{t}(\lambda)}^{\lambda_{i}}
		\dfrac{d\alpha}{\tilde{\varphi}_{i}(\alpha)}.
\end{align*}
Since none of the L\'evy processes $X^{i,i}$ is a subordinator, 
$\lim\limits_{s\rightarrow\infty}\varphi_i(s{\bf 1})=\infty$, for all $i\in[d]$, so that there is $A>0$ such that 
for all $s>A$, $s{\bf 1}\in D_\varphi$. 
Fix $t>0$ and make $\lambda=s{\bf 1}$ tend to infinity. If $u_{t}^{(i)}(\infty)<\infty$ 
and if $(G_{\varphi}^{(i)})$ is not satisfied, then the right hand side of the above inequality tends to 
$\infty$, which is a contradiction. Therefore $u^{(i)}_{t}(\infty)=\infty$ and we derive from Lemma 
$\ref{asympu}$ that for all $t\geq 0$ and $i\in[d]$, $u^{(i)}_{t}(\infty)=\infty$. 
In particular, for all ${\rm r}\in\mathbb{R}_+^{d}\setminus\{\mathbf{0}\}$, $\p_{\rm r}({\rm Z}_t=\mathbf{0})=
e^{-\langle {\rm r},{\rm u}_{t}(\infty)\rangle}=0$. Thus 
$q_{\rm r}=\lim\limits_{t\rightarrow\infty}\p_{\rm r}({\rm Z}_t=\mathbf{0})=0$ and the MCSBP ${\rm Z}$ can only 
be extinguished at infinity. Finally, it follows from Theorem $\ref{1005}$ and (\ref{8226}) that 
\[\bar{q}_{\rm r}:=\p_{\rm r}\left(\lim\limits_{t\rightarrow\infty}{\rm Z}_t=\mathbf{0}\;\;\text{and}\;\;
{\rm Z}_t>\mathbf{0},\;\mbox{for all $t>0$}\right)
=\p_{\rm r}(\lim\limits_{t\rightarrow\infty}{\rm Z}_t=\mathbf{0})=e^{-\langle {\rm r},\phi(\mathbf{0})\rangle},\]
which ends the proof of Theorem $\ref{4588}$. $\qquad\Box$\\

\renewcommand{\theequation}{\Alph{section}.\arabic{equation}}
\appendix{}
\section{Proof of Proposition \ref{8260}}\label{appendix}

Recall that a function $x:\mathbb{R}_+\rightarrow\mathbb{R}^n$, $n\ge1$ is said to be \emph{c\`adl\`ag}, 
if it is right continuous on $\mathbb{R}_+$ and has left limits on $(0,\infty)$. Such a function is said 
to be \emph{downward skip free} if for all $s>0$, $x(s)-x(s-)\ge0$. 
We will use the notation $x_t$ or $x(t)$ indifferently.

\begin{definition}\label{defEd(1.5)}
We call $\mathcal{E}_d$, the set of matrix valued functions $\textsc{x}=\{(x^{i,j}_{t_{j}})_{i,j\in[d]},\,
{\rm t}=(t_{1},\dots, t_{d})\in\mathbb{R}^{d}_{+}\}$ such that
for all $i,j$, $x^{i,j}$ is a c\`adl\`ag function and
\begin{itemize}
\item[$(i)$] $x^{i,j}_0=0$, for all $i,j\in[d]$,
\item[$(ii)$] for all $i\in[d]$, $x^{i,i}$ is downward skip free,
\item[$(iii)$]   for all $i,j\in[d]$ such that  $i\neq j$, $x^{i,j}$ is non-decreasing.
\end{itemize}
\end{definition}

\noindent For ${\rm s}\in \overline{\mathbb{R}}_+^d$, we denote by $[d]_{\rm s}$ the set of indices of finite 
coordinates of ${\rm s}$, that is $[d]_{\rm s}=\{i\in[d]:s_i<\infty\}$. For $i\neq j$, we set 
$x^{i,j}(\infty)=x^{i,j}(\infty-)=\lim\limits_{s\rightarrow\infty} x^{i,j}(s)$.
\begin{definition}\label{(r,x)(1.5)}
Let $\textsc{x}\in\mathcal{E}_d$ and ${\rm r}=(r_1,\dots,r_d)\in\mathbb{R}_+^d$. Then 
${\rm s}\in \overline{\mathbb{R}}_+^d$ is called a solution of the system $({\rm r},\textsc{x})$ if it satisfies
\[({\rm r},\textsc{x})\qquad\qquad r_i+\sum_{j=1}^dx^{i,j}(s_j-)=0\,,\;\;\;i\in [d]_{\rm s}\,.\]
$($In particular, ${\rm s}=(\infty,\infty,\dots,\infty)$ is always a solution of the system 
$({\rm r},\textsc{x})$.$)$
\end{definition}

\noindent Note that in $({\rm r},\textsc{x})$ it is implicit that 
$\sum\limits_{j\in [d]\setminus [d]_{\rm s}}x^{i,j}(s_j-)<\infty$, for all $i\in [d]_{\rm s}$, although by 
definition $s_{j}=\infty$, for $j\in[d]\setminus[d]_{\rm s}$. Let us state Lemma $2.3$ in \cite{cma1} which gives 
the existence and 
uniqueness of a smallest solution to the system $({\rm r},\textsc{x})$.

\begin{lemma}\label{1377} Let $\textsc{x}\in\mathcal{E}_d$ and ${\rm r}=(r_1,\dots,r_d)\in\mathbb{R}_+^d$. 
Then there exists a solution ${\rm s}=(s_1,\dots,s_d)\in\overline{\mathbb{R}}_+^d$ of the system 
$({\rm r},\textsc{x})$ such that any other solution ${\rm t}$ of $({\rm r},\textsc{x})$ satisfies 
${\rm t}\ge {\rm s}$. The solution ${\rm s}$ will be called {\em the smallest solution} of the system 
$({\rm r},\textsc{x})$. 
\end{lemma}

For ${\rm r}=(r_{1},\dots,r_{d})\in\mathbb{R}^{d}_{+}$ and $\textsc{x}\in\mathcal{E}_{d}$, let us now consider 
the functional equation,
\begin{equation}\label{RL}
z^{(i)}_{t} = r_{i} + \sum\limits_{j=1}^{d} x^{i,j}\left(a^{(j)}_t\right), \;\; t\geq 0, i\in[d]\,,
\end{equation}
where $a^{(j)}_t:=\int_0^t z^{(j)}_{s}\,ds$. Existence and uniqueness of solutions of this equation are studied
in Section 3 of \cite{cpgub}. According to a definition in \cite{cpgub}, will say that a solution $z$ 
to (\ref{RL}) has {\it no spontaneous generation} if whenever $z_t^{(i)}=0$ for some $t\ge0$ and all $i$ in some 
subset $I\subset[d]$, the function $s\mapsto a_s^{(i)}$ strictly increases at $t$ for some $i\in I$ only if 
there exists $j\notin I$ such that $x^{i,j}_{a^{(j)}_t}$ increases strictly to the right of $t$. A solution $z$
to (\ref{RL}) is said to be {\it non-negative} if all its coordinates $z^{(i)}$, $i\in[d]$ are non-negative.
Moreover, a non-negative solution $z$ to (\ref{RL}) is said to be {\it conservative} if $z^{(i)}_t<\infty$ for 
all $i\in[d]$ and $t\ge0$. 

\begin{lemma}\label{égalitéps(1.5)}
Let ${\rm r}=(r_{1},\dots,r_{d})\in\mathbb{R}^{d}_{+}$ and $\textsc{x}\in\mathcal{E}_{d}$. Denote by 
${\rm t}_{\rm r}^\textsc{x}$ the smallest solution of the system $({\rm r},\textsc{x})$. Assume that 
$\textsc{x}$ is continuous at time ${\rm t}_{\rm r}^\textsc{x}$ $($i.e. 
$x^{i,j}_{{\rm t}_{\rm r}^{\textsc{x},j}-}=x^{i,j}_{{\rm t}_{\rm r}^{\textsc{x},j}}$, $i,j\in[d]$$)$ and 
that there is a non-negative and conservative solution $z$ to the equation $(\ref{RL})$ which has no spontaneous 
generation. Then
\begin{equation}\label{lien tr int det (1.5)}
{\rm t}^{\textsc{x}}_{\rm r}=\lim_{t\rightarrow +\infty}(a^{(1)}_t,\dots,a^{(d)}_t).
\end{equation}
\end{lemma}
\begin{proof}
Let us first note that $\lim\limits_{t\rightarrow +\infty}{\rm a}_{t}=\lim\limits_{t\rightarrow\infty}(a^{(1)}_t,\dots,a^{(d)}_t)$ is a solution of  the system $({\rm r},\textsc{x})$. 
Indeed, if $a^{(i)}_{\infty}:=\lim\limits_{t\rightarrow\infty}a^{(i)}_t=\infty$ for all $i\in[d]$, then 
$\lim\limits_{t\rightarrow +\infty}{\rm a}_{t}$ is a solution by Definition \ref{(r,x)(1.5)}. 
On the other hand, assume that there is $i\in[d]$ such that $a^{(i)}_\infty=\int_0^\infty z^{(i)}_{s}\,ds<\infty$. 
Since for all $i\neq j$, $x^{i,j}$ is non-decreasing, according to the functional equation (\ref{RL}), 
$z^{(i)}_{\infty}:=\lim\limits_{t\rightarrow\infty}z^{(i)}_{t}$ exists and $z^{(i)}_{\infty}\in[0,\infty]$. 
But $\int_0^\infty z^{(i)}_{s}\,ds<\infty$, so that $z^{(i)}_{\infty}=0$. From (\ref{RL}) and  Definition
\ref{(r,x)(1.5)}, $\lim\limits_{t\rightarrow +\infty}{\rm a}_{t}$ is a solution of the system 
$({\rm r},\textsc{x})$. In particular, from  Lemma \ref{1377}, 
$\lim\limits_{t\rightarrow +\infty}{\rm a}_{t}\ge{\rm t}^{\textsc{x}}_{\rm r}$.

Now observe that the functions $s\mapsto a^{(i)}_s$ are continuous and non decreasing and assume that $t$ is 
such that $t:=\inf\{s:a^{(i)}_s=t^{\textsc{x},i}_{\rm r}\}<\infty$ for some indices $i$ (in particular 
$t^{\textsc{x},i}_{\rm r}<\infty$) and $a^{(j)}_t<t^{\textsc{x},j}_{\rm r}$ for all other indices. With no loss 
of generality, we can assume that $d=2$, $t:=\inf\{s:a^{(1)}_s=t^{\textsc{x},1}_{\rm r}\}<\infty$ and 
$a^{(2)}_t<t^{\textsc{x},2}_{\rm r}$. Since $x^{1,2}$ is non decreasing, 
$x^{1,2}_{a^{(2)}_t-}\le x^{1,2}_{t^{\textsc{x},2}_{\rm r}-}$ and 
$x^{1,2}_{a^{(2)}_t}\le x^{1,2}_{t^{\textsc{x},2}_{\rm r}-}$.
Moreover since  by definition $r_1+x^{1,1}_{t^{\textsc{x},1}_{\rm r}-}+x^{1,2}_{t^{\textsc{x},2}_{\rm r}-}=0$, 
it follows, by the assumption of continuity of $x^{1,1}$ at $t_{\rm r}^{\textsc{x},1}=a^{(1)}_{t}$, that
\begin{equation}\label{6240}
r_1+x^{1,1}_{a^{(1)}_t}+x^{1,2}_{a^{(2)}_t}\le0.
\end{equation}
If $x^{1,2}_{a^{(2)}_t}<x^{1,2}_{t^{\textsc{x},2}_{\rm r}-}$, then it follows from above that $z^{(1)}_{t}<0$, 
which is a contradiction. Therefore $x^{1,2}_{a^{(2)}_t}=x^{1,2}_{t^{\textsc{x},2}_{\rm r}-}$, that is $x^{1,2}$ 
is constant on the interval $[a^{(2)}_t,t^{\textsc{x},2}_{\rm r})$ and $z^{(1)}_{t}=0$.

Since by assumption ${\rm z}$ has no spontaneous generation, it implies that $z^{(1)}$ is absorbed in 0 at time 
$t$, that is $z^{(1)}_{s}=0$, for $s\in[t,t']$, where $t<t'\le\infty$ is such that 
$t'=\inf\{s:a^{(2)}_{s}=t^{\textsc{x},2}_{\rm r}\}$. Since $s\mapsto a^{(1)}_{s}$ is constant on $[t,t']$, it 
follows that $a^{(1)}_{t'}=t^{\textsc{x},1}_{\rm r}$, so that 
$(a^{(1)}_{t'},a^{(2)}_{t'})={\rm t}^{\textsc{x}}_{\rm r}$. This implies, by the assumption of continuity of 
$\textsc{x}$ at time ${\rm t}_{\rm r}^\textsc{x}$, that $z_{t'}=0$ and since $z$ has no 
spontaneous generation, $z_s=0$, for all $s\ge t'$ so that 
$(a^{(1)}_{t'},a^{(2)}_{t'})=\lim_{s\rightarrow\infty} {\rm a}_s$.\\
\end{proof}

\noindent {\it Proof of Proposition $\ref{8260}$}: The proof is a direct application of Lemma 
\ref{égalitéps(1.5)} to the Lamperti representation $(\ref{Lamperti})$. As already mentioned, from Theorem 1
in \cite{cpgub}, for any ${\rm r}\in\mathbb{R}_+^d$ and any family of L\'evy processes $X^{(i)}$, $i\in[d]$ 
defined as in Subsection \ref{8225} and satisfying $\alpha_i=0$, $i\in[d]$, there is a unique solution 
${\rm Z}$ to the equation $(\ref{Lamperti})$. Moreover, as a MCSBP, ${\rm Z}$ is non-negative and its paths 
have no spontaneous generation in the sense defined before Lemma \ref{égalitéps(1.5)}. Moreover, from 
Proposition 3.1 of \cite{cma1}, for  every $i\in[d]$, ${\rm X}^{(i)}$ is a.s. continuous at time 
$T^{(i)}_{\rm r}$. Then on the set $\{{\rm Z}_t\in\mathbb{R}^d_{+},\,t\ge0\}$, which is supposed to be of 
probability 1 here, the paths of the solution ${\rm Z}$ are conservative, in the sense given above. 
Finally Proposition $\ref{8260}$ follows from Lemma \ref{égalitéps(1.5)}. $\;\;\;\Box$.

\end{document}